\documentclass[11pt]{amsart}
\usepackage{a4wide}
\usepackage{amsthm}
\usepackage{amsmath}
\usepackage{amssymb}
\usepackage{amsxtra}
\usepackage[all]{xy}
\usepackage[latin1]{inputenc}

\newtheorem{theorem}{Theorem}[section]
\newtheorem{definition}[theorem]{Definition}

\newtheorem{lemma}[theorem]{Lemma}

\newtheorem{proposition}[theorem]{Proposition}

\newtheorem{remark}[theorem]{Remark}

\newcommand{\C}{\mathbb{C}}

\newcommand{\R}{\mathbb{R}}

\newcommand{\Aut}{\mbox{\rm Aut}}
\newcommand{\Aff}{\mbox{\rm Aff}}
\newcommand{\Der}{\mbox{\rm Der}}

\newcommand{\semi}{\rtimes}

\newcommand{\lieg}{\mathfrak{g}}
\newcommand{\lien}{\mathfrak{n}}
\newcommand{\lieh}{\mathfrak{h}}
\newcommand{\lief}{\mathfrak{f}}

\begin{document}

\title[Affine actions on Nilpotent Lie groups]{Affine actions on Nilpotent Lie groups}

\author[D. Burde]{Dietrich Burde}
\author[K. Dekimpe]{Karel Dekimpe}
\author[S. Deschamps]{Sandra Deschamps}
\address{Fakult\"at f\"ur Mathematik\\
Universit\"at Wien\\
  Nordbergstr. 15\\
  1090 Wien \\
  Austria}
\email{dietrich.burde@univie.ac.at}
\address{Katholieke Universiteit Leuven\\
Campus Kortrijk\\
8500 Kortrijk\\
Belgium}
\date{\today}
\email{Karel.Dekimpe@kuleuven-kortrijk.be}
\email{Sandra.Deschamps@kuleuven-kortrijk.be}

\thanks{The first author thanks the KULeuven Campus Kortrijk  for its hospitality 
and support}
\thanks{The second author expresses his gratitude towards the Erwin Schr\"odinger 
International Institute for Mathematical Physics}
\thanks{Research supported by the Research Programme of the Research 
Foundation-Flanders (FWO): G.0570.06}
\thanks{Research supported by the Research Fund of the Katholieke Universiteit Leuven}

\begin{abstract}
To any connected and simply connected nilpotent Lie group $N$, one can associate its
group of affine transformations ${\rm Aff}(N)$. In this paper, we study simply transitive
actions of a given nilpotent Lie group $G$ on another nilpotent Lie group $N$, via
such affine transformations.

We succeed in translating the existence question of such a simply transitive affine
action to a corresponding question on the Lie algebra level. As an example of the
possible use of this translation, we then consider the case where $\dim(G)=\dim(N)\leq 5$.

Finally, we specialize to the case of abelian simply transitive affine actions on a given
connected and simply connected nilpotent Lie group.
It turns out that such a simply transitive abelian affine action on $N$ corresponds
to a particular Lie compatible bilinear product on the Lie algebra $\lien$ of $N$, which we
call an LR-structure.
\end{abstract}

\maketitle

\section{NIL-affine actions}

In 1977 \cite{miln77-1}, J.~Milnor asked whether or not any connected and simply connected
solvable Lie group $G$ of dimension $n$ admits a representation
$\rho:G\rightarrow \Aff(\R^n)$ into the group of invertible affine mappings,
letting $G$ operate simply transitively on $\R^n$.

For some time, most people were convinced that the answer to
Milnor's question was positive until Y.~Benoist (\cite{beno92-1},
\cite{beno95-1}) proved the existence of a simply connected,
connected nilpotent Lie group $G$ (of dimension 11) not allowing
such a simply transitive affine action. These examples were
generalized to a family of examples by D.~Burde and F.~Grunewald (\cite{bg95-1}),
also in dimension 10 (\cite{burd96-1}).

To be able to construct these counterexamples, both Benoist and Burde--Grunewald used the
fact that the notion of a simply transitive affine action can be translated onto the Lie
algebra level. In fact, if $G$ is a simply connected, connected nilpotent Lie group
with Lie algebra $\lieg$, the existence of a simply transitive affine action of $G$
is equivalent to the existence of a certain Lie algebra representation
\[ \varphi: \lieg \rightarrow \mathfrak{aff}(n)=\R^n \semi \mathfrak{gl}(n,\R).\]

\medskip

As the answer to Milnor's question turned out to be negative, one
tried to find a more general setting, providing a positive answer to
the analogue of Milnor's problem. One such a setting is the setting
of NIL-affine actions. To define this setting, we consider a simply
connected, connected nilpotent Lie group and define the affine group
$\Aff(N)$ as being the group $N\semi \Aut(N)$, which acts on $N$ via
\[ \forall m,n\in N,\;\forall \alpha \in \Aut(N): \;\;^{(m,\alpha)}n=m\cdot \alpha(n).\]
Note that this is really a generalization of the usual affine group
$\Aff(\R^n)=\R^n\semi{\rm GL}(n,\R)$, where ${\rm GL}(n,\R)$ is the
group of continuous automorphisms of the abelian Lie group $\R^n$.
In \cite{deki98-1} it was shown that for any connected and simply
connected solvable Lie group $G$, there exists a connected and
simply connected nilpotent Lie group $N$ and a representation
$\rho:G\rightarrow \Aff(N)$ letting $G$ act simply transitively on
$N$. This shows that in this new setting, any connected and simply
connected solvable Lie group does appear as a simply transitive
group of affine motions of a nilpotent Lie group (referred to as
NIL-affine actions in the sequel).

\medskip

Apart from the existence of such a simply transitive NIL-affine
action for any simply connected solvable Lie group $G$ not much is
known about such actions. As a first approach towards a further
study of this topic, we concentrate in this paper on the situation
where both $G$ and $N$ are nilpotent.

In the following section we show that in this case any simply
transitive action $\rho:G \rightarrow \Aff(N)$ is unipotent. This
result was known in the usual affine case too. Then we obtain
a translation to the Lie algebra level, which is again a very
natural generalization of the known result in the usual affine
situation.

Thereafter we present some examples of simply transitive NIL-affine
actions in low dimensions and finally, we specialize to the
situation where $G$ is abelian.

\section{Nilpotent simply transitive NIL-affine groups are unipotent}

Let $N$ be a connected and simply connected nilpotent Lie group with
Lie algebra $\lien$. It is well known that $N$ has a unique
structure of a real algebraic group (e.g., see \cite{ov93-2}) and
also $\Aut(N)\cong \Aut(\lien)$ carries a natural structure of a
real algebraic group. It follows that we can consider
$\Aff(N)=N\semi \Aut(N)$ as being a real algebraic group. The aim of
this section is to show that any nilpotent simply transitive
subgroup of $\Aff(N)$ is an algebraic subgroup and in fact
unipotent. This generalizes the analogous result for ordinary
nilpotent and simply transitive subgroups of $\Aff(\R^n)$ proved by
J.~Scheuneman in \cite[Theorem 1]{sche75-1}.

Throughout this section we will use $A(G)$ to denote the algebraic
closure of a subgroup $G\subseteq \Aff(N)$ and we will refer to the
unipotent radical of $A(G)$, by writing $U(G)$. Using these
notations, the aim of this section is to prove the following
theorem:

\begin{theorem}\label{uni} Let $N$ be a connected and simply connected nilpotent Lie group and
assume that $G \subset \Aff(N)$ is a nilpotent Lie subgroup acting
simply transitively on $N$, then $G=A(G)=U(G)$.
\end{theorem}

We should mention here that with some effort this theorem can be
reduced from the corresponding theorem in the setting of polynomial
actions as obtained in \cite[Lemma 5]{bd00-1}. For the readers
convenience we will repeat the necessary steps here and adapt them
to our specific situation.

First we recall two basic technical results which will be needed
later on:

\begin{lemma}\label{lema}\cite[Lemma 3]{bd00-1} Let $T$ be a real algebraic torus
acting algebraically on $\R^n$, then the set of fixed points $(\R^n)^T$ is non-empty.
\end{lemma}
Any connected and simply connected nilpotent Lie group $N$ can be identified with
its Lie algebra $\lien$ using the exponential map $\exp$. We can therefore speak of a
polynomial map of $N$, by which we will mean that the corresponding map on the Lie algebra
$\lien$ is expressed by polynomials (with respect to coordinates to any given basis of
$\lien$). For instance, it is well known that the multiplication map
$N\times N\rightarrow N: (n_1,n_2) \mapsto n_1n_2$ is polynomial. We use such polynomial
maps in the following lemma.
\begin{lemma}\label{lemb}\cite[Lemma 2 (a)]{bd00-1}  Let $N$ be a connected and simply
connected nilpotent Lie group. Assume that $\theta: N \times \R^n \to \R^n$ is an action
which is polynomial in both variables. Let $v_0$ be a point in $\R^n$. Then the isotropy
group of $v_0$ is a closed connected subgroup of $N$.
\end{lemma}

Now we adapt \cite[Proposition 1]{bd00-1} to our situation:
\begin{proposition}\label{prop1}
Let $G \subset \Aff(N)$ be a solvable Lie group, acting
simply transitively on $N$. Then the unipotent radical $U(G)$ of
$A(G)$ also acts simply transitively on $N$.
\end{proposition}

\begin{proof}
As $G$ is a connected, solvable Lie subgroup,
its algebraic closure $A(G)$ is solvable also and Zariski connected.
Therefore it splits as a semi-direct product $A(G)=U(G) \semi T$
where $T$ is a real algebraic torus. From the fact that $G$ acts simply transitively,
we immediately get that $A(G)$ acts transitively on $N$.

By Lemma~\ref{lema} and the fact that $N$ is diffeomorphic to $\R^n$ (for $n=\dim N$),
we know that there exists a point $n_0\in N$ which is fixed under the action of $T$.
It follows that $N= (U(G)\semi T)\cdot n_0= U(G)\cdot n_0$, showing that $U(G)$ acts
transitively on $N$.

By \cite[Lemma 4.36]{ragh72-1}, we know that $\dim U(G) \leq \dim
G$. On the other hand $\dim G=\dim N$, as $G$ is acting simply
transitively on $N$, from which we deduce that $\dim U(G)= \dim G =
\dim N$. It follows that $U(G)$ acts with discrete isotropy groups
(because stabilizers have dimension $0$). Lemma~\ref{lemb} then
implies that $U(G)$ acts with trivial isotropy groups, allowing us
to conclude that $U(G)$ acts simply transitively on $N$.
\end{proof}

\medskip

The last step we need before we can prove Theorem~\ref{uni} is the following
(compare with \cite[Lemma 4]{bd00-1})

\begin{lemma}\label{lemc}
Let $G \subset \Aff(N)$ be a solvable Lie group acting simply
transitively on $N $ and let $U(G)$ be the unipotent radical of
$A(G)$. Then the centralizer of $U(G)$ in $A(G)$ coincides with the
center of $U(G)$.
\end{lemma}
\begin{proof} Using Proposition~\ref{prop1} we know that  $U(G)$
acts simply transitively on $N$, and $A(G)$ splits as a semi-direct product $A(G)=U(G) \semi T$
where $T$ is a real algebraic torus.

The centralizer $C$ of $U(G)$ in $A(G)$ is also an algebraic group
and therefore, for every element $c\in C$, the unipotent part $c_u$
and the semisimple part $c_s$ of $c$ are also in $C$. Now assume
that $C$ does not belong completely to $U(G)$, then there is an
element $c$ of $C$ with a nontrivial semisimple part $c_s \neq 1$,
and $c_s$ also belongs to $C$. This semisimple part $c_s$, belongs
to the algebraic torus $T$. So, as in lemma \ref{lema}, we can
conclude that $(N)^{c_s}$ is non-empty.

However, as $c_s$ centralizes $U(G)$, the set $(N)^{c_s}$ is
$U(G)$-invariant. This contradicts the transitivity of $U(G)$ on
$N$. It follows that $C$ is included in $U(G)$.
\end{proof}

\medskip

We are now ready to prove the main result of this section.
\medskip

{\bf Proof of theorem \ref{uni}:}
As $G$ is nilpotent, its algebraic closure $A(G)$ is also nilpotent.
A nilpotent algebraic group splits as a direct product $A(G)=U(G)
\times S(G)$, where $S(G)$ denotes the set of semi-simple elements
of $A(G)$. In this case $S(G)$ centralizes $U(G)$ and by
Lemma~\ref{lemc}, we can conclude that $S(G)$ has to be trivial, so
$A(G)=U(G)$. From the fact that both $G$ and $A(G)=U(G)$ are acting
simply transitively on $N$ (Proposition~\ref{prop1}), we can
conclude that $G=A(G)=U(G)$.
\qed

\section{Translation to the Lie algebra level}

In this section we will show that we can completely translate the notion of a simply
transitive NIL-affine action of a nilpotent Lie group $G$ to a notion on the 
Lie algebra level.

As before, let $G$ and $N$ be connected, simply connected Lie
groups. We will use $\lieg$ to denote the Lie algebra of $G$ and
$\lien$ to denote the Lie algebra of $N$. Recall that the Lie
algebra corresponding to the semi-direct product $\Aff(N)=N\semi
\Aut(N) $ is equal to the semi-direct product $\lien \semi {\rm
Der}(\lien)$, where Der$(\lien)$ is the set of all derivations of
$\lien$. This semi-direct product $\lien \semi {\rm Der}(\lien)$ is
a Lie algebra via
\[ \left[ (X,D), (X',D') \right]=\left([X,X']+ DX'-D'X, [D,D']\right).\]

Assume that $\rho:G \rightarrow \Aff(N)$ is a representation of Lie groups. Then there
exists a unique homomorphism  $d\rho$ (differential of $\rho$)
of their respective Lie algebras $\lieg$ and $\lien \rtimes {\rm
Der}(\lien)$ making the following diagram commutative:
\[
\xymatrix@C=3.5cm{ G \ar[r]^\rho & \Aff(N)\\
\lieg\ar[u]^{\exp} \ar[r]_{d\rho} & \lien \rtimes {\rm Der}(\lien)
\ar[u]_{\exp} }
\]
Conversely, any Lie algebra homomorphism $d\rho:\lieg \rightarrow\lien \rtimes {\rm
Der}(\lien)$   can be seen as the differential of a Lie group homomorphism
$\rho:G \rightarrow \Aff(N)$.

\medskip

Let us use the following notation for $d\rho$:
\[d\rho:\lieg \rightarrow \lien\rtimes {\rm Der}(\lien): X \mapsto d\rho(X)=(t_X, D_X).\]

As $d\rho$ is a Lie algebra homomorphism,
$\mathcal{D}:\lieg \rightarrow {\rm Der}(\lien):X\mapsto D_X$ is also a Lie
algebra homomorphism.

\begin{theorem}\label{algversion}
Let $G$ and $N$ be simply connected, connected nilpotent Lie groups. Let $\rho:G\rightarrow
\Aff(N)$ be a representation. Using the notations introduced above, we have the following:
\begin{center}
$\rho:G\rightarrow \Aff(N)$ induces a
simply transitive NIL-affine action of $G$ on $N$\\
$\Updownarrow$\\
\parbox{8cm}{%
\begin{enumerate}
\item $t:\lieg \rightarrow \lien: X \mapsto t_X$ is a bijection.
\item For any $X\in \lieg$, $D_X$ is nilpotent.
\end{enumerate}}
\end{center}
\end{theorem}

This theorem is known to hold in the usual affine case, i.e., with $N=\R^n$ (see
for example \cite[section 3]{fg83-1}, \cite{kim86-1} and
\cite[p.~100]{dm95-1}), and so we really obtain a very natural generalization.

\begin{proof}
We first prove the direction from the Lie group level (top statement)
to the Lie algebra level (bottom statement). We will proceed by
induction on the dimension $n$ of $G$, where the situation $n=1$ is trivial.\\
So, we assume that $\rho:G\rightarrow \Aff(N)$ induces a simply
transitive NIL-affine action of $G$ on $N$. As $\Aff(N)=N\semi
\Aut(N)$, we can decompose $\rho$ in two component maps
$tr:G\rightarrow N$ and $lin:G\rightarrow \Aut(N)$ with
$\rho(g)=(tr(g), lin(g))$.

The fact that the action is simply transitive, is equivalent to
$tr$ being a bijection between $G$ and $N$, so also $\dim N =  \dim G = n$.

By Proposition~\ref{prop1}, we know that $lin(G)$ is a Lie subgroup of $\Aut(N)$
consisting of unipotent elements. We can therefore choose a  basis
$A_1,\ldots,A_n$ of $\lien$, such that the following two properties hold:

\begin{enumerate}
\item For any $i\in \{1,2,\ldots,n\}$: \ $\lien_i= \langle
A_i,A_{i+1},\ldots, A_n \rangle$ is an ideal of $\lien$. The
corresponding Lie subgroup $N_i=\exp(\lien_i)$ of $N$ is then a normal subgroup
of $N$.\\
If we set $a_j^{x_j}= \exp(x_j A_j)$, for any $x_j \in \R$, then every
element of $N_i$ can be written uniquely in the form $a_i^{x_i} \ldots
a_n^{x_n}$, with $ x_i,\ldots, x_n \in \R$.\\[0.1cm]
\item For any $g \in G$: $lin(g)(N_i) = N_i$, and $lin(g)$ induces the identity on
each quotient $N_i/N_{i+1}$ (with $N_{n+1}=1$).
Seen as an element of $\Aut(\lien)$, the matrix $lin(g)$ is lower triangular
unipotent w.r.t.\ the basis $A_1,A_2\ldots, A_n$.
\end{enumerate}

Inspired by this last property, we introduce $U(N)$ to denote the Lie subgroup of
$\Aut(N)$, consisting of all $\mu\in \Aut(N)$,
for which $\mu(N_i)= N_i$ and for which $\mu$ induces the identity
on each quotient $N_i/N_{i+1}$. So ${lin}(G)\subseteq U(N)$.

$U(N)$ consists of unipotent elements, so denote by
$O(\lien)=\log(U(N))$ the corresponding Lie subalgebra of
$\Der(\lien)$. Every element in $O(\lien)$ is therefore a derivation $D$
for which $D(\lien_i) \subseteq \lien_{i+1}$ for all $i$.

As $N_2\semi U(N)$ is a closed normal subgroup of the Lie group $N \semi
U(N)$, we can construct a Lie group homomorphism $\psi:G \rightarrow
N /N_2\cong \R$ as the composition
\[ G \stackrel{\rho}{\longrightarrow} N\semi U(N)
\longrightarrow \frac{N\semi U(N)}{N_2\semi U(N)}\cong
\frac{N}{N_2}.\] Of course, $\psi:G \rightarrow N/N_2$ can also be
written as the composition of maps
\[ G \stackrel{tr}{\longrightarrow} N \longrightarrow \frac{N}{N_2}.\]
It follows that $\psi $ is onto and that the kernel  of $\psi$ is a
Lie subgroup  $G_2$ of $G$ of codimension 1. We will now use the
induction hypothesis for this Lie subgroup $G_2$. We denote the Lie
subalgebra of $\lieg$ corresponding to $G_2$ by $\lieg_2$. \\[0.2cm]
Note that there is a natural Lie group homomorphism $N_2\semi U(N)
\rightarrow N_2 \semi \Aut(N_2)$ obtained by restricting
automorphisms $\mu \in U(N)$ to automorphisms of $N_2$. As a
conclusion, we find the following diagram  in which each square is
commutative:\\[0.3cm]
\begin{equation}\label{bigdiagram}
\xymatrix{ \lieg_2\ar[rd]_\exp\ar[rr]\ar[dd]_{d\rho'} &  & \lieg\ar[rd]^\exp\ar[dd]^(0.3){d\rho} & \\
& G_2\ar[dd]^(0.3){\rho'}\ar[rr] & & G\ar[dd]^\rho \\
\lien_2 \semi O(\lien)\ar[rd]_\exp\ar[rr]\ar[dd] & & \lien\semi \Der(\lien)\ar[rd]^\exp & \\
& N_2 \semi U(N)\ar[dd]\ar[rr] & & N\semi \Aut(N) \\
\lien_2\semi \Der(\lien_2)\ar[rd]_\exp & & & \\
& N_2\semi \Aut(N_2) & & \\
}
\end{equation}
\medskip
In the above diagram, $\rho'$ (resp.\ $d\rho'$) is obtained from
$\rho$ (resp.\ $d\rho$) by restricting both the domain and the
codomain. Now, the composite map
\[ G_2 \stackrel{\rho'}{\longrightarrow} N_2\semi U(N) \longrightarrow
N_2 \semi \Aut(N_2) \] satisfies the condition from the statement of
the theorem. Therefore, we can use the induction hypothesis to
conclude that the composite map
\[ \lieg_2 \stackrel{d\rho'}{\longrightarrow} \lien_2\semi O(\lien) \rightarrow \lien_2\semi \Der(\lien_2),  \]
with
\[ X_2 \mapsto (t_{X_2},D_{X_2}) \mapsto (t_{X_2},D_{X_2}^{'}),\]
(where $D_{X_2}^{'}$ denotes the restriction of $D_{X_2}$ to $\lien_2$) satisfies the following conditions:
\begin{enumerate}
\item $t:\lieg_2 \rightarrow \lien_2 \subseteq \lien: X_2 \mapsto t_{X_2}$ is a bijection.
\item For any $X_2 \in \lieg_2$, $D_{X_2}^{'}$ is nilpotent (evidently, as $D_{X_2} \in O(\lien)$).
\end{enumerate}

\medskip

Now, consider the map $d\rho:\lieg \rightarrow \lien \semi
\Der(\lien)$.
As $\rho(G) \subseteq N \semi U(N)$, we get that $d \rho(\lieg )
\subseteq \lien \semi O(\lien)$, so for any $X\in \lieg$, $D_X$ is
evidently nilpotent.

\medskip

Now, fix an element $b_1\in G$, for which $\psi(b_1)=a_1 \cdot N_2$.
Take $B_1=\log(b_1)$, and we already know that $A_1=\log(a_1)$. Then
we have that \[ \lien= \langle A_1 \rangle + \lien_2 \mbox{ and }
\lieg= \langle B_1 \rangle + \lieg_2.\]

To show that $t: \lieg \to \lien$ is bijective, take an arbitrary
element of $\lien$.  Such an element is of the form
$s A_1 + m_2$, for some $s \in \R$ and some $m_2\in \lien_2$.

From the fact that $\psi(b_1) = a_1 \cdot N_2$, one deduces that:
\[ d\rho(B_1) = (A_1 +  m_2^{'}, D),\mbox{ for some } m_2^{'} \in \lien_2\mbox{ and }D\in O(\lien).\]
It follows that
\[t_{ s B_1}= s A_1 + s m_2^{'} \]

As $t: \lieg_2 \to \lien_2 \subseteq \lien$ is a bijection, and $m_2
- s m_2^{'} \in \lien_2$, there exists a $g_2\in \lieg_2$ such that
\[ t_{g_2}= m_2 - s m_2^{'}. \]
So $t$ maps $s B_1 + g_2$ to $s A_1 + m_2$, showing that $t$ is
surjective. In an analogous way, one shows that $t$ is injective,
from which we conclude that $t: \lieg \to \lien$ is a bijection.

\bigskip

We now prove the direction from the Lie algebra level (bottom statement)
to the Lie group level (top statement). We will proceed by
induction on the dimension $n$ of $\lieg$, where the situation $n=1$ is trivial.\\
So, we assume that

\begin{enumerate}
\item $t:\lieg \rightarrow \lien: X \mapsto t_X$ is a bijection. So
$\dim  \lien = \dim \lieg = n$.
\item For any $X\in \lieg$, $D_X$ is nilpotent.
\end{enumerate}

As $\mathcal{D}(\lieg)$ is a Lie subalgebra of $\Der(\lien)$ consisting of nilpotent elements,
we can choose a vector space basis
$A_1,\ldots, A_n$ of $\lien$, such that

\begin{enumerate}
\item for any $i\in \{1,2,\ldots,n\}$: \
$\lien_i= \langle A_i,A_{i+1},\ldots, A_n \rangle$ is an ideal of $\lien$, and
\item for any $X\in \lien$: $D_X(\lien_i) \subseteq
\lien_{i+1}$, where $\lien_{n+1}=0$ (i.e., the matrix $D_X$ is strictly lower triangular
w.r.t.\ the basis $A_1,A_2\ldots, A_n$).
\end{enumerate}

Now we use $O(\lien)$ to denote the Lie subalgebra of $\Der(\lien)$ consisting of all derivations
$D\in \Der(\lien)$ for which $D(\lien_i) \subseteq
\lien_{i+1}$ for all $i$. So $\mathcal{D}(\lieg)\subseteq O(\lien)$.

As the $\lien_i$ are ideals in the Lie algebra $\lien$, the corresponding Lie subgroups
$N_i=\exp(\lien_i)$ are closed normal subgroups of $N$. They determine a filtration of closed subgroups
\[ N=N_1\supset N_2 \supset \cdots \supset N_n\supset N_{n+1}=1 \]
where $N_i/N_{i+1}\cong \R$.

Let $\exp(O(\lien))=U(N)$ be the Lie subgroup of $\Aut(N)$  corresponding to $O(\lien)$, then for any
element $\mu\in U(N)$, we have that $\mu(N_i)= N_i$ and $\mu$ induces the identity on each quotient
$N_i/N_{i+1}$.

\medskip

As $\lien_2\semi O(\lien)$ is obviously an ideal of the Lie algebra
$\lien\semi O(\lien)$, we can construct a Lie algebra homomorphism
$\psi:\lieg \rightarrow \lien/\lien_2\cong \R$ as the composition
\[ \lieg \stackrel{d\rho}{\longrightarrow} \lien\semi O(\lien)
\longrightarrow \frac{\lien\semi O(\lien)}{\lien_2\semi O(\lien)}\cong \frac{\lien}{\lien_2}.\]
Of course, $\psi:\lieg \rightarrow \lien/\lien_2$ can also be written as the composition of linear maps
\[ \lieg \stackrel{t}{\longrightarrow}\lien \longrightarrow \frac{\lien}{\lien_2}.\]
It follows that $\psi $ is onto and that the kernel  of $\psi$ is an ideal $\lieg_2$ of $\lieg$ of
codimension 1. We will now use the induction hypothesis for this ideal $\lieg_2$. We denote the
subgroup of $G$ corresponding to $\lieg_2$ by $G_2$.

Note that there is a natural Lie algebra homomorphism $\lien_2\semi O(\lien) \rightarrow
\lien_2 \semi \Der(\lien_2)$ obtained by restricting derivations $D\in O(\lien)$ to derivations of
$\lien_2$.  As a conclusion, we again obtain diagram (\ref{bigdiagram}) in which each square is commutative.

\medskip

Now, the composite map
\[ \lieg_2 \stackrel{d\rho'}{\longrightarrow} \lien_2\semi O(\lien) \longrightarrow
\lien_2 \semi \Der(\lien_2) \]
satisfies conditions 1.\ and 2.\ from the statement of the theorem. Therefore, we can use the
induction hypothesis to conclude that the composite map
\[ G_2 \stackrel{\rho'}{\longrightarrow} N_2\semi U(N) \rightarrow N_2\semi \Aut(N_2) \]
determines a simply transitive action of $G_2$ on $N_2$. It follows that the map $\rho:G\rightarrow
N\semi \Aut(N)$ describes an action of $G$ on $N$, for which the subgroup $G_2$ acts on $N$ in such a
way that the subgroup $N_2$ is a single orbit for the $G_2$-action and $G_2$ acts simply transitively
on $N_2$.

\medskip

Now, fix an element $B_1\in \lieg$, for which $\psi(B_1)=A_1+\lien_2$. Then we use the notation
$b_1^t=\exp(t B_1)\in G$ (resp.\ $a_1^t=\exp(t A_1)\in N$) for all $t\in \R$. Now, we can consider
the two 1-parameter subgroups
\[ B=\{ b_1^t\;|\; t\in \R\}\cong \R \mbox{ and } A=\{ a_1^t\;|\; t\in \R\}\cong \R\]
of $G$ and $N$ respectively, for which we have that
\[ G=G_2 \cdot B \mbox{ and } N=A\cdot N_2.\]
To show that $G$ acts transitively on $N$, it is enough to show that
any element of $N$ is in the same orbit as the identity. An
arbitrary element of $N$ is of the form $a_1^t\cdot n_2$, for some
$t\in \R$ and some $n_2\in N_2$. From the fact that $\psi(B_1) =
A_1+\lien_2$, one deduces that for a given $t\in \R$:
\[ \rho(b_1^{-t}) = (a_1^{-t}\cdot  n_2', \mu),\mbox{ for some } n_2'\in N_2\mbox{ and }\mu\in U(N).\]
It follows that
\[ ^{\rho(b_1^{-t})} a_1^t\cdot n_2= n_2''\mbox{ for some } n_2'' \in N_2.\]
As $G_2$ acts simply transitively on $N_2$, there exists a $g_2\in
G_2$ such that
\[ ^{\rho(g_2)}n_2''=1 \]
and so $g_2\cdot b_1^{-t}$ maps $a_1^t\cdot n_2$ to the identity.
This shows that the action of $G$ is transitive. Analogously, one
shows that $g_2\cdot b_1^{-t}$ is the unique element mapping
$a_1^t\cdot n_2$ to the identity, from which we conclude that the
action is simply transitive.
\end{proof}

\section{The situation in low dimensions}

Now that we have obtained a description of simply transitive NIL-affine actions
on the Lie algebra level, we are able to study the existence of such actions
in low dimensions. As a first result, we obtain the following proposition.

\begin{proposition}
Let G and N be connected, simply connected nilpotent Lie groups of
dimension $n$ with $1 \leq n \leq 5$. Then there exists a
representation $\rho:G\rightarrow \Aff(N)$ which induces a simply
transitive NIL-affine action of $G$ on $N$.
\end{proposition}

\begin{proof} Let $\lieg$ and $\lien$ be the corresponding Lie
algebras.  By theorem \ref{algversion} we know that it is enough to
show that there exists a Lie algebra homomorphism $d\rho:\lieg
\rightarrow \lien\rtimes {\rm Der}(\lien): X \mapsto d\rho(X)=(t_X,
D_X)$ , for which the translational part $t:\lieg \rightarrow \lien
: X \mapsto t_X$ is a bijection and all $D_X$ are nilpotent.

For each Lie algebra we will fix a basis $X_1,\ldots, X_n$
of the underlying vector space and denote the coordinate
of an element $\sum_{i=1}^{n} x_i X_i $ by the column vector
$( x_1, x_2 ,\ldots,x_n)^T$. We will sometimes represent a
derivation $D_X$ by its matrix with respect to the given basis $X_1,\ldots,X_n$.

Let us first remark that if $\lien=\lieg$, with $\dim(\lieg)=n$, then
there exists a trivial Lie algebra morphism: \[d\rho_{triv}:\lieg
\rightarrow \lieg\rtimes {\rm Der}(\lieg): X \mapsto
d\rho_{triv}(X)=(X, D_0(n)),\] where $D_0(n)$ is the trivial
nilpotent derivation which maps every element of $\lieg$ to $0$.
Of course, this representation corresponds to the simply transitive
action of $G$ on itself, using left translations.

So up to dimension $3$ we only have to search for suitable Lie
algebra morphisms between $\R^3$ and $\lieh_3$, the 3-dimensional Heisenberg Lie algebra.
Let $\R^3$ be the Lie algebra with  basis vectors $X_1,X_2,X_3$ and all Lie brackets
equal to zero. Let $\lieh_3$ be the Lie algebra with
basis vectors $X_1,X_2,X_3$ and non-zero brackets
$[X_1,X_2]=X_3$. Then one can choose the following suitable Lie
algebra morphisms (expressed in coordinates):

\[ d\rho:\R^3
\rightarrow \lieh_3\rtimes {\rm Der}(\lieh_3): X=( x_1, x_2 ,x_3)^T
\mapsto \left(( x_1, x_2 ,x_3)^T , D_X= \left( \!
\begin{array}{ccc}  0 & 0 & 0
\\ 0 &0&0\\\frac{ x_2}{2} & \frac{-x_1}{2}  &0\end{array} \! \right) \right),\]

and

 \[ d\rho:\lieh_3
\rightarrow \R^3\rtimes {\rm Der}(\R^3): X=( x_1, x_2 ,x_3)^T
\mapsto \left(( x_1, x_2 ,x_3)^T, D_X= \left( \!
\begin{array}{ccc}  0 & 0 & 0
\\ 0 &0&0\\\frac{ -x_2}{2} & \frac{x_1}{2}  &0\end{array} \! \right)   \right).\]

In dimension $4$ there are $3$ different nilpotent Lie algebras over $\R$.
Denoting the basis by $X_1,X_2,X_3,X_4$, let $\R^4$ be the Lie algebra
with all Lie brackets equal to zero, $\lieh_3
\oplus \R$ be the Lie algebra with non-zero bracket $[X_1,X_2]=X_3$, and
$\lief_4$ be the filiform nilpotent Lie algebra with non-zero
brackets $[X_1,X_2]=X_3, [X_1,X_3]=X_4$.

We can find suitable Lie algebra morphisms $d\rho:\lieg
\rightarrow \lien\rtimes {\rm Der}(\lien): X=( x_1, x_2 ,x_3,x_4)^T
\mapsto d\rho(X)=(t_X, D_X)$ for all possible combinations of
$\lieg$ and $\lien$ in dimension $4$.
In every case we can choose $t_X=( x_1, x_2 ,x_3,x_4)^T$.
The following table shows possible choices of linear parts:

\medskip

 \begin{tabular}{|c|c|c|c|} \hline
 & ${\rm Der}(\R^4)$ & ${\rm Der}(\lieh_3 \oplus \R)$ & ${\rm Der}(\lief_4)$\\
\hline

$\R^4$ &  $ D_0(4)$ &    $\left(\! \begin{array}{cccc} 0 & 0 & 0&0
\\ 0 &0&0&0 \\ \frac{x_2}{2 }&\frac{-x_1}{2}&0&0 \\  0 & 0
& 0 & 0
\end{array} \! \right)  $  &  $ \left(\! \begin{array}{cccc} 0 & 0 & 0&0
\\ 0 &0&0&0 \\ 0&-x_1&0&0 \\  0 & 0
& -x_1 & 0
\end{array} \! \right)  $    \\ \hline

$\lieh_3 \oplus \R$ &   $ \left(\! \begin{array}{cccc} 0 & 0 & 0&0
\\ 0 &0&0&0 \\ \frac{-x_2}{2 }&\frac{x_1}{2}&0&0 \\  0 & 0
& 0 & 0
\end{array} \! \right)  $  &  $ D_0(4) $  & $\left(\! \begin{array}{cccc} 0 & 0 & 0&0
\\ -x_1 &0&0&0 \\ 0&0&0&0 \\  x_3 & x_2
& 0 & 0
\end{array} \! \right)  $ \\ \hline

$\lief_4$ &  $ \left(\! \begin{array}{cccc} 0 & 0 & 0&0
\\ 0 &0&0&0 \\ 0&x_1&0&0 \\  0 & 0
& x_1 & 0
\end{array} \! \right) $     &  $ \left( \! \begin{array}{cccc}  0 & 0 &
0&0
\\ x_1 &0&0&0 \\ x_1+x_2&x_1&0&0 \\  3 x_1-x_3 & x_2
& 0 & 0  \end{array} \! \right)$ &  $ D_0(4) $  \\ \hline

\end{tabular}

\vspace*{0.5cm}

Analogously one can show that the theorem also holds in dimension
$5$. Up to isomorphism there are $9$ nilpotent Lie algebras over
$\R$, so one has to find in total $81$ suitable Lie algebra
morphisms satisfying the properties of Theorem~\ref{algversion}.
A full list of those Lie algebra morphisms is available from the authors.
As an example we treat here one case. Let $\lieh_{3} \oplus \R^2$
be the Lie algebra with basis vectors $X_1,\ldots, X_5$ and non-zero
brackets $[X_1,X_2]=X_3 $ and let $\lieg_{5,6}$ (using the notation
of \cite{magn95-1}) be the Lie algebra with basis vectors
$X_1,\ldots, X_5$ and non-zero brackets $[X_1,X_2]=X_3; \
[X_1,X_3]=X_4 ; \ [X_1,X_4]=X_5 ; \ [X_2,X_3]=X_5$.
Then a suitable Lie algebra morphism is the following:
\[ d\rho:\lieh_{3} \oplus \R^2 \rightarrow \lieg_{5,6} \rtimes {\rm Der}(\lieg_{5,6}):
X=( x_1, x_2, x_3,x_4,x_5)^T \mapsto \left(
\left(x_1,x_2,\frac{x_3}{\sqrt{3}},x_4,x_5\right)^T,D_X \right), \]
with
\[ D_X=
\left(
\begin{array}{lllll}
 0 & 0 & 0 & 0 & 0 \\
 \frac{x_1}{-3+\sqrt{3}} & 0 & 0 & 0 & 0 \\
 x_2 & \frac{x_1}{\sqrt{3}} & 0 & 0 & 0 \\
 \frac{1}{3} \left(1+\sqrt{3}\right) x_3 & x_2-\frac{x_2}{\sqrt{3}} & \frac{x_1}{\sqrt{3}} & 0 & 0 \\
 \frac{1}{6} \left(3+\sqrt{3}\right) x_4 & \frac{1}{3} \left(-1+\sqrt{3}\right) x_3 & -\frac{x_2}{\sqrt{3}} &
   \frac{1}{6} \left(-3+\sqrt{3}\right) x_1 & 0
\end{array}
\right).\]

\medskip
\noindent
We remark here that for most other cases, the representation is of a simpler form than
the one above and can be obtained by rather simple computations.
\end{proof}

\begin{proposition}\label{counterexample}
Consider the nilpotent Lie algebra $\lien$ of dimension $6$, with
basis $X_1,\ldots,X_{6}$ and non-trivial Lie brackets
\[
\begin{array}{lll}
{[X_1,X_2]}  = X_3, & [X_1,X_3] = X_4,  & [X_1,X_4]=X_5, \\
{[X_2,X_5]}  = X_6, & [X_3,X_4]=-X_{6}.  &
\end{array}
\]
and let  $N$ be the corresponding Lie group. Then $\R^6$ cannot act simply transitive
by NIL-affine actions on $N$.
\end{proposition}

\begin{proof}
We will assume that there exists a representation $\rho: \R^6
\rightarrow \Aff(N)$ which induces a simply transitive
NIL-affine action of $\R^6$ on $N$, and show that this leads to a
contradiction. By theorem \ref{algversion} we know that the differential
$d\rho:\R^6 \rightarrow \lien\rtimes {\rm Der}(\lien): X
\mapsto d\rho(X)=(t_X, D_X)$ is a Lie algebra morphism, for which
the translational part $t:\R^6 \rightarrow \lien : X \mapsto
t_X$ is a bijection and all $D_X$ are nilpotent.
As $t$ is a bijection, there exist $Y_i \in \R^6$ for which
$t_{Y_i}=X_i$. So we have $d\rho(Y_i)=(X_i,D_i)$, with $D_i \in {\rm
Der}(\lien)$. One can check (or consult \cite{magn95-1}, where $\lien$ equals the
Lie algebra $\lieg_{6,18}$) that the matrix of any
derivation $D_i, \; i=1,\ldots,6$, with respect to the basis
$X_1,\ldots, X_6$ is of the form:

\[ D_i=\left(
\begin{array}{llllll}
 \alpha_{i} & 0 & 0 & 0 & 0 & 0 \\
 0 & \beta_{i} & 0 & 0 & 0 & 0 \\
 \gamma_{i1} & \gamma_{i2} & \alpha_{i}+\beta_{i} & 0 & 0 & 0 \\
 \delta_{i} & 0 & \gamma_{i2} & 2 \alpha_{i}+\beta_{i} & 0 & 0 \\
 \epsilon_{i1} & \epsilon_{i2} & 0 & \gamma_{i2} & 3 \alpha_{i}+\beta_{i} & 0 \\
 \varphi_{i1} & \varphi_{i2} & -\epsilon_{i1} & \delta_{i} & -\gamma_{i1} & 3 \alpha_{i}+2 \beta_{i}
\end{array}
\right).\] As $d\rho$ is a Lie algebra morphism, we know that
$d\rho([Y_i,Y_j])=d\rho(0)=$
\begin{equation}\label{counter}
(0,0)=[(X_i,D_i),(X_j,D_j)]=( [X_i,X_j] +
D_i(X_j)-D_j(X_i),[D_i,D_j]).
\end{equation}
It follows that
\[ 0=  [X_i,X_j] +D_i(X_j)-D_j(X_i), \mbox{ for all }(i,j) \mbox{ with } 1\leq i < j \leq 6.\]
Considering the above for the pairs $(i,j)$ equal to $(1,3)$, $(1,4)$, $(2,3)$ and $(3,4)$
we obtain a system of linear equations in the entries of the $D_i$ implying that
\[
 \epsilon_{41}= \frac{1}{2},\;
\gamma_{12}= \frac{-1}{2},\; \delta_{3}=\frac{1}{2},\;
\alpha_3=\beta_3=\gamma_{32}=\epsilon_{31}=0.\]
From these conditions it follows that $[D_1,D_2]\neq 0$, which contradicts (\ref{counter}).
\end{proof}

\begin{remark}
We will later on, based on Theorem~\ref{thm-LR}, obtain a more conceptual proof of 
this result. In fact, such an action can only exist if the Lie algebra $\lien$ is two-step
solvable, see \cite{bdd07-2}, section $2$.
\end{remark}

\section{Abelian simply transitive groups and LR-structures}

Thus far we have been looking to simply transitive actions
$\rho:G\rightarrow \Aff(N)$, where both $G$ and $N$ are arbitrary real connected and
simply connected nilpotent Lie
groups. As already pointed out in the introduction, the case where $N\cong \R^n$ has been
well studied and is equivalent to the study of (complete) left symmetric structures on the
Lie algebra $\lieg$, corresponding to the Lie group $G$.

In this section, we deal with the case that $N$ is arbitrary and $G$ is abelian.
It turns out that this case is equivalent to the study of complete LR-structures (see below
for a definition) on the Lie algebra $\lien$.

In fact, the main result of this section is the following

\begin{theorem}\label{thm-LR}
Let $N$ be a connected and simply connected nilpotent Lie group of dimension $n$. 
Then there exists a simply transitive NIL-affine action of $\R^n$ on $N$ via a 
representation $\rho\colon \R^n \rightarrow \Aff (N)$ if and only if  the Lie 
algebra $\lien$ of $N$ admits a complete LR-structure.
\end{theorem}

Before giving the proof let us define the notion of an LR-algebra and
an LR-structure.

\begin{definition}
An algebra $(A,\cdot)$ over a field $k$ with product $(X,Y) \mapsto X\cdot Y$
is called {\it LR-algebra}, if the product satisfies
the identities
\begin{align}
X\cdot (Y\cdot Z)& = Y\cdot (X\cdot Z) \label{lr1}\\
(X\cdot Y)\cdot Z& =(X\cdot Z)\cdot Y \label{lr2}
\end{align}
for all $X,Y,Z \in A$.
\end{definition}

If we denote by $L(X)$, $ R(X)$ the left, respectively right multiplication operator in
the algebra  $(A,\cdot)$, then the above conditions say that all left-multiplications
and all right multiplications commute:
\begin{align*}
[L(X),L(Y)] & = 0,\\
[R(X),R(Y)] & = 0.
\end{align*}

Note that LR-algebras are Lie-admissible algebras: the commutator defines
a Lie bracket.

The associated Lie algebra then is said to admit an LR-structure:

\begin{definition}\label{l-r}
An {\it LR-structure} on a Lie algebra
$\lieg$ over $k$ is an LR-product $\lieg \times \lieg \rightarrow \lieg$
satisfying
\begin{equation}\label{lr3}
[X,Y]=X\cdot Y -Y\cdot X
\end{equation}
for all $X,Y \in \lieg$. It is said to be {\it complete}, if all left
multiplications $L(X)$ are nilpotent.
\end{definition}

We now come to the proof of Theorem~\ref{thm-LR}.
\begin{proof}
Suppose that $\rho\colon \R^n\rightarrow \Aff (N)$
induces a simply transitive NIL-affine action of $\R^n$ on $N$.
Theorem~\ref{algversion} says that this is equivalent to the existence
of a Lie algebra homomorphism
\[d \rho: \R^n \rightarrow \lien \semi \Der(\lien): X \mapsto (t_X, D_X),
\]
where $t_X:\R^n\rightarrow \lien$ is a linear isomorphism and each $D_X$ is
nilpotent. As $t_X$ is bijective, we identify $\R^n$ and $\lien$ as vector spaces and hence
we can write $d\rho (X)=(X,D_X)$. The fact that $d\rho$ is a homomorphism from the abelian
Lie algebra to the Lie algebra $\lien \semi \Der(\lien)$ is equivalent to the 
requirement that for all $X,Y\in \lien$
\begin{align*}
0 & = [(X,D_X),(Y,D_Y)] \\
  & = ([X,Y] + D_X(Y)-D_Y(X),[D_X,D_Y]).
\end{align*}
Define a bilinear product on $\lien$ by $X\cdot Y= -D_X(Y)$.
We will show that this defines a complete LR-structure on $\lien$. In fact, the above
condition is equivalent to the conditions
\begin{align*}
[X,Y] & = X\cdot Y -Y\cdot X, \\
X\cdot (Y\cdot Z)  & = Y\cdot (X\cdot Z).
\end{align*}
We have to show that also the right multiplications commute. Since $D_X$ is
a derivation of $\lien$ we have
\begin{align*}
X\cdot (Z\cdot Y) -X\cdot (Y\cdot Z) & = X\cdot [Z,Y] \\
 & = - D_X([Y,Z]) \\
 & = -[D_X(Y),Z]-[Y,D_X(Z)] \\
 & = (X\cdot Y)\cdot Z -Z\cdot (X\cdot Y)+Y\cdot (X \cdot Z)-(X\cdot Z)\cdot Z
\end{align*}
for all $X,Y,Z\in \lien$. Since the left multiplications commute,
this is equivalent to $(X\cdot Y)\cdot Z=(X\cdot Z)\cdot Y$, showing
that also the right multiplications commute, and hence this bilinear  product
determines an LR-structure on $\lien$.

By Theorem~$\ref{algversion}$ we know that all derivations $D_X$, and hence all
left multiplications $L(X)$ are nilpotent. This proves that the LR-structure is
also complete and so we have obtained the first direction of our result.
The converse direction follows in a similar way.
\end{proof}

Theorem~\ref{thm-LR} shows that a deeper study of (complete) LR-algebras
is required in order to obtain a good understanding of simply transitive
abelian and NIL-affine actions on nilpotent Lie groups. The first steps
in this direction are taken in \cite{bdd07-2}.


\end{document}